\documentclass[11pt]{amsart}

\usepackage{amsmath,amssymb,latexsym,soul,cite,mathrsfs,amsfonts}

\usepackage{color,enumitem,graphicx}
\usepackage[colorlinks=true,urlcolor=blue,
citecolor=red,linkcolor=blue,linktocpage,pdfpagelabels,
bookmarksnumbered,bookmarksopen]{hyperref}
\usepackage[english]{babel}
\usepackage{comment}

\usepackage[left=2.9cm,right=2.9cm,top=2.8cm,bottom=2.8cm]{geometry}
\usepackage[hyperpageref]{backref}

\usepackage[colorinlistoftodos]{todonotes}
\makeatletter
\providecommand\@dotsep{5}
\def\listtodoname{List of Todos}
\def\listoftodos{\@starttoc{tdo}\listtodoname}
\makeatother

\usepackage{verbatim} % to using the command "\comment"

\numberwithin{equation}{section}
\newtheorem{theorem}{Theorem}[section]
\newtheorem{proposition}[theorem]{Proposition}
\newtheorem{lemma}[theorem]{Lemma}
\newtheorem{corollary}[theorem]{Corollary}
\newtheorem{remark}{Remark}
\newcommand\restr[2]{{% we make the whole thing an ordinary symbol
  \left.\kern-\nulldelimiterspace % automatically resize the bar with \right
  #1 % the function
  \vphantom{\big|} % pretend it's a little taller at normal size
  \right|_{#2} % this is the delimiter
  }}
\title[Fibering method for Schr\"odinger type systems]
{The fibering method approach for a \\  non-linear Schr\"odinger equation coupled \\ with the 
electromagnetic field
}

\author[G. Siciliano]{Gaetano Siciliano}

\author[K. Silva]{ Kaye Silva}

\address[G. Siciliano]{\newline\indent
	Departamento de Matem\'atica - Instituto de Matem\'atica e Estat\'istica
	\newline\indent 
	Universidade de S\~ao Paulo
	\newline\indent
	Rua do Mat\~ao 1010,  05508-090  S\~ao Paulo, SP,  Brazil}
\email{\href{mailto:sicilian@ime.usp.br}{sicilian@ime.usp.br}}

\address[K. Silva]{\newline\indent
	Instituto de Matem\'atica e Estat\'istica.   
	\newline\indent 
	Universidade Federal de Goi\'as,
	\newline\indent
	74001-970, Goi\^ania, GO, Brazil}
\email{\href{mailto:kayeoliveira@hotmail.com}{kayeoliveira@hotmail.com}}

\thanks{
Gaetano Siciliano was partially supported by Fapesp, CNPq and Capes, Brazil. 
}

\begin{document}

\begin{abstract}
We study, with respect to the parameter $q\neq0$, the following  Schr\"odinger-Bopp-Podolsky system
in $\mathbb R^{3}$
\begin{equation*}
\left\{
\begin{aligned}
-&\Delta u+\omega u+q^2\phi u=|u|^{p-2}u, \\
&-\Delta \phi+a^2\Delta^2 \phi = 4\pi u^2,
\end{aligned}
\right.
\end{equation*}
where $p\in(2,3], \omega>0, a\geq0$ are fixed.
We prove, by means of the fibering approach, that the system
has no solutions {\sl at all} for large values of $q's$,
and has two {\sl radial } solutions for small $q's$. We give also 
qualitative properties about  the energy level of the solutions and 
 a variational characterization of these  extremals values of $q$.
Our results  recover and improve some results in \cite{GaeDa, Ruiz}.
\end{abstract}

\maketitle

\section{Introduction}

In the recent paper \cite{GaeDa} for the first time in the mathematical
literature the following system in $\mathbb R^{3}$ has been studied
\begin{equation}\label{p}
\left\{
\begin{aligned}
-&\Delta u+\omega u+q^2\phi u=|u|^{p-2}u, \\
&-\Delta \phi+a^2\Delta^2 \phi = 4\pi u^2,
\end{aligned}
\right.
\end{equation}
where $a,\omega> 0$, $q\neq0$ and $p\in (2,6)$. The system appears when one\ look for stationary
solutions $u(x)e^{i\omega t}$ of the Schr\"odinger equation coupled with the Bopp-Podolski 
Lagrangian of the electromagnetic field, in the purely electrostatic situation.
Here $u$ represents the modulus of the wave function and $\phi$ the electrostatic potential.
From a physical point of view, the parameter $q$ has the meaning of the electric charge and $a$ is the parameter of
the Bopp-Podolski term.

In the cited paper, it has been shown that the problem can be addressed variationally. Indeed introducing the Hilbert space  $$\mathcal{D}:=\left\{\phi\in {D}^{1,2}(\mathbb{R}^3):\ \Delta \phi\in L^2(\mathbb{R}^3)\right\}$$
normed by $$\|\phi\|^{2}_{\mathcal D} = a^{2}\|\Delta \phi\|_{2}^{2} + \|\nabla \phi \|_{2}^{2},$$
it can be proved that % the test functions are dense, 
for every $u\in H^{1}(\mathbb R^{3})$ there is a unique solution $\phi_{u}\in \mathcal D$ of the second equation
in the system, that is satisfying
	\begin{equation}\label{solu}
-\Delta\phi_u+	a^2\Delta^2\phi_u=4\pi u^2.
	\end{equation}
Moreover it turns out that
$$\phi_{u} = \frac{1 - e^{-|\cdot|/a}}{|\cdot|} *u^{2}.$$
Observe from \eqref{solu} that, for every $ u\in H^{1}(\mathbb R^{3})$,
\begin{equation*}%\label{norm1}
4\pi \int \phi_u u^2=\|\phi_u\|^2_{\mathcal{D}},
\end{equation*}
which will be used throughout the paper.

By using the classical by now {\sl reduction argument }
one is led to study, equivalently,  the single equation
\begin{equation}\label{eq:equacao}
-\Delta u+\omega u+q^2\phi_{u} u=|u|^{p-2}u \quad \text{ in } \ \mathbb R^{3}
\end{equation}
containing the  nonlocal term $\phi_{u} u$.
Then whenever from now on we speak of solution of the system \eqref{p} we mean just the solution $u$
of the above equation since $\phi=\phi_{u}$ is univocally determined.
To the equation \eqref{eq:equacao} is related the energy functional
\begin{equation*}%\label{eq:J}
\mathcal{J}_q(u)=\frac{1}{2}\|u\|^2+\frac{q^2}{4}\int \phi_u u^2-\frac{1}{p}\|u\|_p^p, \quad u\in H^{1}(\mathbb R^{3})
\end{equation*}
which is well defined and $C^{1}$. In this way we are simply reduced to find critical points of $\mathcal J_{q}$.
%The reader can find all these details in \cite{GaeDa}.
We are denoting (here and throughout the paper) by $\|u \|_{p}$ the $L^{p}-$norm and by 
$$\|u \|^{2} =\|\nabla u\|^{2} + \omega\|u\|_{2}^{2}$$
the (squared) norm in $H^{1}(\mathbb R^{3})$, being $\omega$ a fixed positive constant.

\medskip

In \cite[Theorem 1.1 and Theorem 1.2]{GaeDa} it is proved that
if $p\in (3,6)$, then problem \eqref{p} admits a solution for every $q\neq0$.  On the other hand, if $p\in(2,3]$
the existence of a solution is proved just for $q\neq0$ sufficiently small. 
As we can see, there is a difference in the result depending on the range where $p$ varies.
Indeed in the case of $p's$ ``small'' the value of $q$ may prevent the existence of critical points 
for the functional $\mathcal J_{q}$.

\medskip

Of course, if $a=0$  system \eqref{p} reduces to the so called Schr\"odinger-Poisson system
in $\mathbb R^{3}$
\begin{equation}\label{eq:sp}
\left\{
\begin{aligned}
-&\Delta u+\omega u+q^2\phi u=|u|^{p-2}u, \\
&-\Delta \phi = 4\pi u^2
\end{aligned}
\right.
\end{equation}
or, equivalently, to
$$-\Delta u+\omega u+q^2\phi_{u}^{\textrm SP} u=|u|^{p-2}u, $$
where  now
$$\phi_{u}^{\textrm SP} = \frac{1}{|\cdot|}* u^{2}%\in D^{1,2}(\mathbb R^{3})
.$$
In the mathematical literature there is a huge number of papers concerning this last problem.
However we cite here just \cite{BF} where the authors for the first time introduced
the {\sl reduction method} which allows to study a single equation
instead of a system, and \cite{Ruiz} where the author studies the problem depending on the parameter
$q^{2}$. In particular Ruiz in \cite{Ruiz}, among other results, shows that, in the case $p\in(2,3)$
 the system \eqref{eq:sp} has two  radial solutions for small $q^{2}$ and has no solutions 
 at all, that is radial or not, for $q^{2}\ge 1/4$. See also \cite{RS} for similar results related to the problem in bounded domains.

\bigskip

Motivated by the cited papers \cite{GaeDa,Ruiz}, our aim here is to understand in a more satisfactory way
the existence of solutions for \eqref{p}, or \eqref{eq:equacao}, namely
the behaviour of $\mathcal J_{q}$ for what concerns its critical points in the case $p\in(2,3]$
and how they are influenced by the value of  $q$.

We prove two type of results.
The first one concerns with the smallness of $q^{2}$ as a necessary condition
in order to have a nontrivial solution of the problem (the sufficiency being proved in \cite{GaeDa}).
Indeed we show that for $q^{2}$ suitably large the problem has no solutions at all.
See Theorem \ref{th:1} below.

The second result concerns the existence of solutions for $q^{2}$ small. 
However, due to the technique used (we borrow some ideas from \cite{Ruiz}),
 we are able to state such a result just in the radial
case: in this case obtaining two solutions 
(in spite of the result in \cite{GaeDa} which states the existence of one solution in the nonradial case).
See Theorem \ref{th:2} below

\medskip

Before stating rigorously the results, we
observe that in the problem is appearing the positive parameter $q^{2}$. 
In view of this, the results are  stated and proved
for simplicity just for $q>0$, being understood that they are valid by changing $q$ with $|q|$.
%By a solution of \eqref{p} we mean a critical point of the $C^1$ functional $\mathcal{J}_q$. 
More specifically, under the assumption $p\in(2, 3]$ our main results in this work are the following
\begin{theorem}\label{th:1}
There exists $q^{*}>0$ such that, for every $q>q^{*}$ the problem admits only the trivial solution.
\end{theorem}

\begin{theorem}\label{th:2}
There exist $\varepsilon>0,q_0^{*}>0$ satisfying $q_{0}^{*}+\varepsilon<q^{*}$
(with $q^{*}$ given in Theorem \ref{th:1})
such that, for every   $q\in (0,q_0^*+\varepsilon)$ the problem
has two radial solutions.
\end{theorem}

Few comments on these results are in order.

As we already said, we are reduced to find critical points of $\mathcal J_{q}$.
We remark explicitly  that no Pohozaev identity is involved in proving the nonexistence result in Theorem \ref{th:1}:
it just follows due to the properties of the fibering map.

To prove  Theorem \ref{th:2} we will use the Mountain Pass Theorem
on the space of radial functions, that is in $H^{1}_{r}(\mathbb R^{3})$.
We take advantage of the smallness of $q$ to prove that the energy functional has 
the Mountain Pass Geometry. However, in contrast to \cite{GaeDa} where the
condition of $q$ small was used in order to find a function where $\mathcal J_{q}$
is negative (and then apply the Mountain Pass Theorem in a standard way), the value $q_{0}^{*}$ we find here is a threshold: for $q<q_{0}^{*}$
there is  a function in $H^{1}_{r}(\mathbb R^{3})$ where the functional is negative, while for $q\geq q_{0}^{*}$
the functional is non-negative. Hence  the argument employed in both papers \cite{GaeDa,Ruiz} do not work
for  $q> q_0^*$, nevertheless also in this case we exhibit here a Mountain Pass structure.

As shown in \cite{GaeDa}, the solutions we find in Theorem \ref{th:2} are classical
and positive by  the Maximum Principle.

Our results holds for any fixed $a\geq0$.
We also notice that, in case $a=0$
we do not need Lemma \ref{regu} and the inequality in  Proposition \ref{ine1}
just follows (up to some positive constant)
by multiplying the second equation in \eqref{eq:sp} by $|u|$ and integrating, which is the relation used in \cite{Ruiz}.
In this sense our result recovers the one in \cite{Ruiz} and  even a better understanding of the fiber 
maps is given here, since the Mountain Pass structure for the functional related to \eqref{eq:sp}
holds although the functional is  non-negative for $q\geq q_{0}^{*}$.

Actually  we deduce Theorem \ref{th:1}
and Theorem \ref{th:2} as consequences of the following result which gives
additional information on the solutions.
\begin{theorem}\label{T1} 
%There exist $q^*>q_0^*>0$ and $\varepsilon>0$ such that for each $q\in (0,q_0^*+\varepsilon)$ 
Let $p\in (2,3], a\geq0$ be fixed. There exist positive numbers $\varepsilon, q_{0}, q_{0}^{*}$ satisfying $q_{0}^{*}+\varepsilon<q^{*}$ such that:
\begin{enumerate}
\item[1.] for each $q>q^{*}$ the functional $\mathcal J_{q}$ has no critical points in $H^{1}(\mathbb R^{3})$
other than the zero function; \medskip
\item[2.]for each $q\in (0,q_0^*+\varepsilon)$ the functional $\mathcal J_{q}$ has two nontrivial critical points $u_{q},w_{q}\in H^{1}_{r}(\mathbb R^{3})$ where $w_{q}$ is a Mountain Pass critical point with
 $$\mathcal{J}_q(w_q)>\max\{ 0, \mathcal{J}_q(u_q)\}.
 %\quad \text{ and }\quad \mathcal{J}_q(w_q)>\mathcal{J}_q(u_q).
 $$
More specifically, \medskip
	\begin{enumerate}[label=(\roman*),ref=\roman*]
		\item\label{i1T1} if $q\in (0,q_0^*]$, then $u_q$ is a global minimum, with
$$\mathcal{J}_q(u_q)<0\ \text{ if }\  q\in (0,q_0^*), \quad \mathcal J_{q_{0}^{*}}(u_{q_{0}}^{*})=0;$$
				\item\label{i2T1} if $q\in (q_0^*,q_0^*+\varepsilon)$, then $u_q$ is a local minimum
				with
				$$\mathcal{J}_q(u_q)>0.$$
				% \medskip
		%\item\label{i3T1} $\mathcal{J}_q(u_q)<0$ if $q\in (0,q_0^*)$, $\mathcal{J}_q(u_{q_0^*})=0$ and $
		%\mathcal{J}_q(u_q)>0$ if $q\in (q_0^*,q_0^*+\varepsilon)$, 
		%while $\mathcal{J}_q(w_q)>0$ and $\mathcal{J}_q(w_q)>\mathcal{J}_q(u_q)$ for each $q\in (0,q_0^*+
		%\varepsilon)$. 
		%\item\label{i4T1} If $q>q^*$, then the functional $\mathcal{J}_q$ has no 
		%critical points other than the zero function.
	\end{enumerate}
\end{enumerate}
\end{theorem}

\medskip

As we can see, whenever $q=q_{0}^{*}$, then $\mathcal J_{q_{0}^{*}}$ is non-negative and we find a global
minimiser at zero energy; then an additional work is needed in order to show that this is not the zero function.
This result is new also in the case $a=0$.

Concerning the  {\sl extremal values} $q_{0}^{*}$ and $q^{*}$ we say here they have a variational
characterisation (see Section \ref{sec:prelim}).
Furthermore, although we were not able to prove it, it seems plausible that for all $q\in (0,q^*)$, the system \eqref{p} has two positive solutions satisfying the properties in Theorem \ref{T1} and $q^*$ is in fact a bifurcation parameter where the two solutions (the local minimum and the Mountain Pass type solution) collapses. 

\medskip

We point out finally that similar results have been obtained in some nonlinear problems depending on a parameter 
in some recent papers: see Il'yasov and Silva \cite{YaKa}, Silva and Macedo \cite{KaAb}.

\medskip

This paper is organised as follows. 

In Section \ref{sec:prelim}
some preliminaries and technical results (true in the general nonradial setting) are given.
As a byproduct of these results, the proof of item 1. in Theorem \ref{T1} follows, see 
Corollary \ref{nehari}.

In Section \ref{sec:PS} the important Proposition \ref{MPS} is proved. It concerns with radial 
functions and is fundamental in proving our result.

Finally, the proof of Theorem \ref{T1} is completed in Section \ref{sec:final}.

\medskip

As a matter of notations, we use the generic letters $C,C',\ldots$ to denote a positive
constant, usually related to Sobolev embedding, whose value may also change
from line to line: no confusion should arise.

\section{Preliminaries and technical  results}\label{sec:prelim}

In \cite{GaeDa} some properties of the solution $\phi_{u}$ are found.
 However to deal with the case $p\in(2,3]$
we need also the next ones. Of course this applies just for $a\neq0.$

\begin{lemma}\label{regu}
 For each $u\in H^1(\mathbb{R}^3)$ we have
 \begin{enumerate}[label=(\roman*),ref=\roman*]
 \item\label{1regu} $\Delta\phi_u\in \mathcal D$, \medskip
 \item\label{2regu} $a^{2}\Delta\phi_u\le \phi_u.$
 \end{enumerate}

\end{lemma}

\begin{proof} 
Let us fix $u\in H^{1}(\mathbb R^{3})$ and let $\psi_{u}: = -\Delta \phi_{u}$. Then 
\begin{equation*}
-a^{2}\Delta \psi_{u} +\psi_{u}= 4\pi u^{2}
\end{equation*}
Since $u^{2}\in L^{2}(\mathbb R^{3})$, by standard results 
%(see e.g. Livro Kesavan Th. 3.3.1)
we have $\psi_{u}\in H^{2}(\mathbb R^{3})$
and in particular, 
$$\psi_{u}\in L^{6}(\mathbb R^{3}), \  \nabla \psi_{u} \in L^{2}(\mathbb R^{3}) \quad \text {and } \quad  \Delta \psi_{u}\in L^{2}(\mathbb R^{3}).$$
This gives	 $\Delta\phi_{u} \in  D^{1,2}(\mathbb R^{3})$ and $\Delta^{2}\phi_{u}\in L^{2}(\mathbb R^{3})$, namely
$\Delta\phi_{u}\in \mathcal D$ proving \eqref{1regu}.

On the other hand, if we set $v=-a^2\Delta\phi_u+\phi_u$,  then
	\begin{equation*}
	-\Delta v=4\pi u^2 \ge 0
	\end{equation*}
	and $v\in {D}^{1,2}(\mathbb{R}^3)$ and is continuous.
	Define $\Omega^-=\{x\in \mathbb{R}^3:\ v(x)<0\}$ and  suppose that $\Omega^-\neq \emptyset$. Once $v$ is continuous, the set $\Omega^-$ is  open. Let $v^-=\max\{-v,0\}$. It follows that 
	\begin{equation*}
	-\int _{\Omega^-}|\nabla v|^2=\int \nabla v\nabla v^-\ge 0,
	\end{equation*}
	which is a contradiction, therefore, $\Omega^-=\emptyset$ and $a^{2}\Delta\phi_u\le {\phi_u}$ in $\mathbb R^{3}$, proving \eqref{2regu}.
\end{proof}
%
%\begin{lemma}\label{maxi} For each  $u\in H^1(\mathbb{R}^3)$, there holds 
%	\begin{equation*}
%	\Delta\phi_u\le \frac{\phi_u}{a^2}.
%	\end{equation*}
%\end{lemma}
%\begin{proof} Indeed, if we set $v=-a^2\Delta\phi_u+\phi_u$,  then
%	\begin{equation*}
%	-\Delta v=4\pi u^2 \ge 0
%	\end{equation*}
%	and $v\in {D}^{1,2}(\mathbb{R}^3)$ and is continuous.
%	Define $\Omega^-=\{x\in \mathbb{R}^3:\ v(x)<0\}$ and  suppose that $\Omega^-\neq \emptyset$. Once $v$ is continuous, the set $\Omega^-$ is  open. Let $v^-=\max\{-v,0\}$. It follows that 
%	\begin{equation*}
%	-\int _{\Omega^-}|\nabla v|^2=\int \nabla v\nabla v^-\ge 0,
%	\end{equation*}
%	which is a contradiction, therefore, $\Omega^-=\emptyset$ and $\Delta\phi_u\le {\phi_u}/{a^2}$ in $\mathbb R^{3}$.
%	\end{proof}

The next result will be useful to get a generalisation of \cite[Formula (19)]{Ruiz} to the case $a\neq0$.
 \begin{proposition}\label{ine1} There holds
	\begin{equation*}
\int |u|^3\le \frac{1}{\pi}\|\phi_u\|_{\mathcal{D}}\|\nabla u\|_2, \quad \forall\,  u\in H^1(\mathbb{R}^3).
	\end{equation*}
\end{proposition}

\begin{proof} %If $a=0$ the result is trivial. If $a>0$, note that 
%	For every $u\in H^{1}(\mathbb R^{3})$ it is
%	\begin{equation}\label{diffeq}
%-\Delta\phi_u+	a^2\Delta^2\phi_u=4\pi u^2.
%	\end{equation}
For $u\in H^{1}(\mathbb R^{3})$ fixed, let us consider equation \eqref{solu}.
	 Since by  Lemma \ref{regu} it holds in particular that $\nabla \Delta\phi_{u}\in L^{2}(\mathbb R^{3})$,
	 by multiplying the  equation \eqref{solu} by $|u|\in H^{1}(\mathbb R^{3}) $ and integrating we get
	\begin{eqnarray}\label{diffeq1}
	4\pi\int |u|^3&=&a^2\int \nabla (-\Delta\phi_u)\nabla |u|+\int \nabla \phi_u\nabla|u| \nonumber\\
	&\le& a^2\|\nabla(-\Delta\phi_u)\|_2\|\nabla u\|_2+\|\nabla \phi_u\|_2\|\nabla u\|_2 \nonumber\\
	&=&\left(a^2\|\nabla(-\Delta\phi_u)\|_2+\|\nabla \phi_u\|_2\right)\|\nabla u\|_2\nonumber \\
	&\le& \left(a^2\|\nabla(\Delta\phi_u)\|_2+\|\phi_u\|_{\mathcal{D}}\right)\| \nabla u\|_2.
	\end{eqnarray}
	On the other hand by multiplying  \eqref{solu} by $\Delta\phi_{u}\in \mathcal D$
        and  making use of   \eqref{2regu} of Lemma \ref{regu} we get
	\begin{eqnarray*}%\label{diffeq2}
	a^2\|\nabla (\Delta\phi_u)\|_2^2&=&4\pi\int \Delta\phi_u u^2-\int \nabla \phi_u \nabla(\Delta\phi_u) \nonumber\\
	&\le&  \frac{1}{a^{2}}\|\phi_u\|_{\mathcal{D}}^2+\frac{\varepsilon^2}{2} \|\nabla(\Delta\phi_u)\|_2^2+\frac{1}{2\varepsilon^2}\|\nabla \phi_u\|_2^2 \nonumber\\
	&\le&  \frac{1}{a^{2}}\|\phi_u\|_{\mathcal{D}}^2+\frac{\varepsilon^2}{2} \|\nabla(\Delta\phi_u)\|_2^2+\frac{1}{2\varepsilon^2}\| \phi_u\|_{\mathcal{D}}^2.
	\end{eqnarray*}
	By choosing  $\varepsilon=a$ above %\eqref{diffeq2} 
	we conclude that, for all $u\in H^{1}(\mathbb R^{3})$,
	\begin{equation}\label{diffeq3}
	a^2\|\nabla (\Delta\phi_u)\|_2^2\le\frac{2}{a^{2}}\|\phi_u\|_{\mathcal{D}}^2+\frac{1}{a^2}\|\phi_u\|_{\mathcal{D}}^2=\frac{3}{a^2}\|\phi_u\|_{\mathcal{D}}^2.
	\end{equation}
	From \eqref{diffeq1} and \eqref{diffeq3} we conclude that 
	\begin{equation*}
	\int |u|^3\le \frac{1}{\pi}\|\phi_u\|_{\mathcal{D}}\|\nabla u\|_2,
	\end{equation*}
	completing the proof.
\end{proof}
%\begin{proposition}\label{awayzero} Suppose that $\mathcal{J}'_q(u)=0$, then there exists a positive constant $C_q$ such that $\|u\|\ge C_q$.
%\end{proposition}
%\begin{proof} From $\mathcal{J}'_q(u)=0$ we have that 
%	\begin{equation*}
%	1+q^2\|u\|^2\left\|\phi_{\frac{u}{\|u\|}}\right\|_\mathcal{D}^2-\|u\|^{p-2}\left\|\frac{u}{\|u\|}\right\|_p^p=0,
%	\end{equation*}
%	which completes the proof.	
%\end{proof}
We conclude this section by showing a first simple property of the energy functional.
The next result says that the functional $\mathcal J_{q}$ has a strict local minimum in $0$,
uniformly in $q$. However to have the complete Mountain Pass structure
$q$ has to be small, as it will be shown in Corollary \ref{zeroenergylevel}.
\begin{proposition}\label{MPGAN} There exist $\rho>0$ and $M>0$ such that 
$$\forall q\in \mathbb R, u \in H^{1}(\mathbb R^{3}) \text{ with } \|u\|=\rho :  \  \mathcal{J}_q(u)\ge M.$$
\end{proposition}
\begin{proof} Since
	\begin{equation*}
	\mathcal{J}_q(u)\ge \frac{1}{2}\|u\|^2-\frac{1}{p}\|u\|_p^p\ge \frac{1}{2}\|u\|^2-C \|u\|^p,
	\end{equation*}
	the conclusion easily follows.
\end{proof}

%\section{Technical Results}
%For $u\in H^1(\mathbb{R}^3)$, consider the reduced functional 
%
%\begin{equation*}
%\mathcal{J}_q(u)=\frac{1}{2}\|u\|^2+\frac{q^2}{4}\int \phi_u u^2-\frac{1}{p}\|u\|_p^p,
%\end{equation*}
%where $\|u\|^2=\|u\|_2^2+\|\nabla u\|_2^2$ and $\phi_u\in \mathcal{D}_a\equiv\{\phi\in \mathcal{D}^{1,2}(\mathbb{R}^3):\ a^2\Delta \phi\in L^2(\mathbb{R}^3)\}$ is the unique solution of 
%\begin{equation}\label{solu}a^2\Delta^2 \phi-\Delta\phi=4\pi u^2.
%\end{equation}
%The Sobolev space $\mathcal{D}$ is equipped with the norm 
%\begin{equation}\label{norm}
%\|\phi\|_{\mathcal{D}}^2=\|\nabla \phi\|_2^2+a^2\|\Delta \phi\|_2^2,\ \forall\ \phi\in \mathcal{D}.
%\end{equation}
%Observe from \eqref{solu} that
%\begin{equation}\label{norm1}
%4\pi \int \phi_u u^2=\|\phi_u\|^2_{\mathcal{D}},\quad  \forall\ u\in H^{1}(\mathbb R^{3}).
%\end{equation}

In this Section we establish some notations and study the geometry of the functional $\mathcal{J}_q$.  We observe that $\phi_{tu}=t^2\phi_u$ and therefore, if $\psi_{q,u}:[0,\infty)\to \mathbb{R}$ is defined by $\psi_{q,u}(t)=\mathcal{J}_q(tu)$, we have that 
\begin{equation*}
\psi_{q,u}(t)=\frac{t^2}{2}\| u\|^2+\frac{q^2t^4}{4}\int \phi_u u^2-\frac{t^p}{p}\|u\|_p^p.
\end{equation*}
Whenever $q$ and $u$ are fixed, we will use for brevity also the notation $\psi:=\psi_{q,u}$.
A simple analysis shows that
\begin{proposition}\label{fiberingmaps} For each $q\in\mathbb{R}\setminus\{0\}$ and $u\in H^1(\mathbb{R}^3)\setminus\{0\}$, there are only three possibilities for the graph of $ \psi$:
	\begin{enumerate}[label=(\roman*),ref=\roman*]
		\item\label{i1fibering} the function $\psi$ has only two critical points when $t>0$, to wit, $0<t_q^-(u)<t_q^+(u)$. Moreover, $t_q^-(u)$ is a local maximum with $\psi''(t_q^-(u))<0$ and $t_q^+(u)$ is a local minimum with $\psi''(t_q^+(u))>0$; \medskip
		\item\label{i2fibering} the function  $\psi$ has only one critical point when $t>0$ at the value $t_q(u)$. Moreover, $\psi''(t_q(u))=0$ and  $\psi$ is increasing;  \medskip
		\item\label{i3fibering} the function  $\psi$ is increasing and has no critical points.
	\end{enumerate}
\end{proposition}
It is important to notice that \eqref{i1fibering} happens for $q$ small, and \eqref{i3fibering} for $q$ large.

Let us consider the Nehari manifold associated with the functional $\mathcal{J}_q$, that is
\begin{equation*}
\mathcal{N}_q=\{u\in H^1(\mathbb{R}^3)\setminus\{0\}:\ \psi_{q,u}'(1)=0\}.
\end{equation*}
Note that, for $u\in \mathcal N_{q}$:
$$\|u\|^{2}\leq \|u\|^{2} +q^{2}\int \phi_{u} u^{2}\leq C \|u\|^{p}$$
and then all the Nehari manifolds are bounded away from zero uniformly in $q$, in the sense that
\begin{equation}\label{eq:bddaway}
\exists\, \widetilde C>0 \ \text{ such that for all }q\in \mathbb R, \ u\in \mathcal N_{q} : \|u\|\geq \widetilde C.   
\end{equation}

%\begin{proposition}\label{prop:Nehari} It holds:
%\begin{enumerate}[label=(\roman*),ref=\roman*]
%\item\label{i1Nehari} there exists a constant $C>0$ such that, 
%$$\text{ for all }q\in \mathbb R, u\in\mathcal N_{q}  : \|u\|\geq C,$$ 
%\item\label{i2Nehari} for all $q\in \mathbb R$,
%$$\inf_{\mathcal N_{q}} J_{q} >0.$$
%\end{enumerate}
%\end{proposition}

Moreover since
\begin{equation*}
\mathcal{N}_q=\mathcal{N}_q^+\cup \mathcal{N}_q^0\cup \mathcal{N}_q^-, 
\end{equation*}
where 
\begin{equation*}
\mathcal{N}_q^+=\{u\in \mathcal{N}_q:\ \psi''(1)>0\},
\end{equation*}
\begin{equation*}
\mathcal{N}_q^0=\{u\in \mathcal{N}_q:\ \psi''(1)=0\},
\end{equation*}
\begin{equation*}
\mathcal{N}_q^-=\{u\in \mathcal{N}_q:\ \psi''(1)<0\}.
\end{equation*}
as an application of the Implicit Function Theorem one has the following:
\begin{proposition}\label{minnehari} If $\mathcal{N}_q^+,\mathcal{N}_q^-\neq \emptyset$, then $\mathcal{N}_q^+,\mathcal{N}_q^-\neq \emptyset$ are $C^1$ manifolds of codimension $1$ in $H^1(\mathbb{R}^3)$. Moreover, $u\in \mathcal{N}_q^+\cup\mathcal{N}_q^-$ is a critical point for the functional $\mathcal{J}_q$ if and only if $u$ is a critical point of the constrained functional $(\mathcal{J}_q)_{|\mathcal{N}_q^+\cup\mathcal{N}_q^-}: \mathcal{N}_q^+\cup\mathcal{N}_q^-\to \mathbb{R}$.
\end{proposition}

Note that, fixed $u\in H^{1}(\mathbb R^{3})\setminus\{0\}$ we have $tu\in \mathcal{N}_q^0$ if, and only if 
$\psi'_{q,tu}(1) = \psi''_{q,tu}(1)=0$
, i.e. the following system of equations is satisfied:
\begin{equation}\label{extremal}
\left\{
\begin{aligned}
%\psi'_{q,u}(t)=
&t\|u\|^2+q^2t^3\int \phi_u u^2-t^{p-1}\|u\|_p^p = 0, \\
%\psi''_{q,u}(t)=
 &\|u\|^2+3q^2t^2\int \phi_u u^2-(p-1)t^{p-2}\|u\|_p^p = 0.
\end{aligned}
\right.
\end{equation}
We can solve the system \eqref{extremal} with respect to the variables $q$ and $t$ to obtain a unique solution given by
\begin{equation*}%\label{eq:}
t(u) =\Big( \frac{2\|u\|^{2}}{(4-p) \|u\|_{p}^{p}}\Big)^{1/(p-2)}
\end{equation*}
and
\begin{equation}\label{extremalfun}
q(u)=C_p\frac{\|u\|_p^{p/(p-2)}}{\|u\|^{(4-p)/(p-2)} \|\phi_u\|_\mathcal{D}},
\qquad C_{p}=\frac{2 (p-2)^{1/2} \pi^{1/2}(4-p)^{(4-p)/2(p-2)}}{2^{1/(p-2)}}.
\end{equation}
%where $C_p=$ is a positive constant depending only on $p$.
In addition the solutions $q(u)$ and $t(u)$ are related by
\begin{equation*}%\label{tu}
t(u)=\left(\frac{2q^2(u)}{4\pi(p-2)}\frac{\|\phi_{u}\|_{\mathcal D}^{2}}{\|u\|_p^p}\right)^{{1}/(p-4)}.
\end{equation*}
 Define the extremal value (see Il'yasov \cite{ilyasENMM})
\begin{equation*}
q^*=\sup\left\{q(u):\ u\in H^1(\mathbb{R}^3)\setminus\{0\} \right\}.
\end{equation*}

\begin{lemma}\label{extremalparamater} The function $H^1(\mathbb{R}^3)\setminus\{0\}\ni u\mapsto q(u)$ defined in \eqref{extremalfun} is $0$-homogeneous. Moreover $q^*<\infty$. 
\end{lemma}

\begin{proof} That $q(u)$ is zero homogeneous is obvious. Let us prove that $q^*<\infty$. Indeed, once $p\in (2,3]$ we have from the interpolation inequality that, for all $u\in H^{1}(\mathbb R^{3})$ we have 
	\begin{equation}\label{rt}
	\|u\|_p^p\le \|u\|^{6-2p}\|u\|_3^{3p-6}.
		\end{equation}
Combining the inequality \eqref{rt} with the Proposition \ref{ine1} we conclude that
\begin{equation}\label{rt1}
\|u\|_p^p\le  C\|u\|^{6-2p}\|u\|^{(3p-6)/3}\|\phi_u\|_\mathcal{D}^{(3p-6)/3}=C\|u\|^{4-p}\|\phi_u\|_\mathcal{D}^{p-2}.
\end{equation}
for some constant $C>0$. It follows from \eqref{rt1} that 
\begin{equation*}
q(u)\le C\frac{\|u\|^{(4-p)/(p-2)}\|\phi_u\|_\mathcal{D}}{\|u\|^{(4-p)/(p-2)} \|\phi_u\|_\mathcal{D}}\le C
\end{equation*}
completing the proof.
\end{proof}

Another extremal value  which will be important for us is the one such that, for larger values of the parameter, the functional
is always non-negative. Let us start by   fixing $u\in H^{1}(\mathbb R^{3})\setminus\{0\}$ and
considering the system
\begin{equation*}%\label{zeroene}
\left\{
\begin{aligned}
\psi_{q_0,u}(t_0)= \frac{t_0^2}{2}\|u\|^2+ q_0^2\frac{t_0^4}{4}\int \phi_u u^2-\frac{t_0^{p}}{p}\|u\|_p^p &= 0, \\
\psi'_{q_0,u}(t_0)= t_0\| u\|^2+q_0^2t_0^3\int \phi_u u^2-t_0^{p-1}\|u\|_p^p &= 0.
\end{aligned}
\right.
\end{equation*}
One can solve this system with respect to the variables $t_0$ and $q_0$ to obtain the unique solution given by
\begin{equation}\label{zeroenergy}
q_0(u)=C_{0,p}\frac{\|u\|_p^{p/(p-2)}}{\|u\|^{(4-p)/(p-2)} \|\phi_u\|_\mathcal{D}}, \qquad
C_{0,p} =\frac{2^{3/2} (p-2)^{1/2} \pi^{1/2}(4-p)^{(4-p)/2(p-2)}}{p^{1/(p-2)}}
\end{equation}
and $t_{0}(u)$ is given by
\begin{equation*}
t_0(u)=\left(\frac{pq_0^2(u)}{2(p-2)}\frac{\|\phi_u\|_{\mathcal D}^2}{\|u\|_p^p}\right)^{1/(p-4)}.
\end{equation*}
Observe that $C_{0,p}<C_p$, where $C_{p}$
is the one appearing in \eqref{extremalfun}. Then $q_{0}(u)< q(u)$. Define the extremal value as
\begin{equation*}
q_0^*=\sup\left\{q_0(u):\ u\in H^1(\mathbb{R}^3)\setminus\{0\} \right\}.
\end{equation*}

\begin{remark}\label{extremalparameter1}Once $q_0(u)$ is a multiple of $q(u)$, Lemma \ref{extremalparamater} also holds true for the function $q_0$.
\end{remark}

The solutions $q(u)$ and $q_0(u)$ given in \eqref{extremalfun}  and \eqref{zeroenergy} have the following geometrical interpretation which can be proved starting from the Proposition \ref{fiberingmaps}.
\begin{proposition}\label{fiberingvariation} For each $u\in H^1(\mathbb{R}^3)\setminus\{0\}$ there holds:
	\begin{enumerate}[label=(\roman*),ref=\roman*]
		\item\label{i1} $q(u)$ is the unique parameter $q>0$ for which the fiber map $\psi_{q,u}$ has a critical point with second derivative zero at $t(u)$. Moreover, if $0<q<q(u)$, then $\psi_{q,u}$ satisfies \eqref{i1fibering} of the Proposition \ref{fiberingmaps} while if $q>q(u)$, then $\psi_{q,u}$ satisfies \eqref{i3fibering} of the Proposition \ref{fiberingmaps}. \medskip
		\item\label{i2} $q_0(u)$ is the unique parameter $q>0$ for which the fiber map $\psi_{q,u}$ has a critical point with zero energy at $t_0(u)$. Moreover, if $0<q<q_0(u)$, then $\inf_{t>0}\psi_{q,u}(t)<0$ while if $q>q_{0}(u)$, then $\inf_{t>0}\psi_{q,u}(t)=0$.
	\end{enumerate}
\end{proposition}

Moreover the parameter $q_0^*$ has the geometrical interpretation that for each $q\in (0,q_0^*)$, there exists at least one $ u\in H^1(\mathbb{R}^3)\setminus\{0\}$ for which $\mathcal{J}_q(u)<0$, while if $q\ge q_0^*$, then $\mathcal{J}_q(u)\ge 0$ for all $ u\in H^1(\mathbb{R}^3)$. In  both works \cite{GaeDa,Ruiz}  the necessity of small values of $q$ was imposed in order to show that there exists a function where the functional is negative, in such a way that
$\mathcal{J}_q$ possesses a Mountain Pass Geometry. 
Therefore, the argument employed in both papers do not work
for  $q> q_0^*$.

The above proposition has the following important consequences. 
\begin{corollary}\label{nehari}
	 %Suppose that $p\in (2,3)$, then 
	 If $q>q^*$ the functional $\mathcal{J}_q$ has no critical points other then the zero function. Moreover if $q<q^*$, then $\mathcal{N}_q^-\neq\emptyset$ and $\mathcal{N}_q^+\neq\emptyset$.
\end{corollary}

In particular item 1. of Theorem \ref{T1} is proved.
\begin{proof} 
%Indeed, we claim that if $q>q^*$, for each $u\in H^{1}(\mathbb R^{3})\setminus\{0\}$
%the function $\psi_{q,u}$ has no critical points.  This is a consequence of the inequality $q(u)\le q^*<q$ and the %Proposition \ref{fiberingvariation}.
It is sufficient to show that, for each  $u\in H^{1}(\mathbb R^{3})\setminus\{0\}$,
the function $\psi_{q,u}$ has no critical points for $q>q^{*}$.
Actually this is a consequence of the inequalities $q(u)\le q^*<q$ and 
\eqref{i1} of  Proposition \ref{fiberingvariation}.

Now assume that $q<q^*$. From the definition of $q^*$, there exists $u\in H^1(\mathbb{R}^3)\setminus\{0\}$ such that $q<q(u)<q^*$. Therefore, from \eqref{i1} of Proposition \ref{fiberingvariation} we conclude that $\mathcal{N}_q^-\neq\emptyset$ and $\mathcal{N}_q^+\neq\emptyset$.
\end{proof}

\begin{corollary}\label{zeroenergylevel} For each $q\ge q_0^*$, there holds $\mathcal{J}_q(u)\ge 0$ for all $u\in H^1(\mathbb{R}^3)$. Moreover, if $q<q_0^*$, then there exists $u\in H^1(\mathbb{R}^3)$ such that $\mathcal{J}_q(u)<0$.
\end{corollary}

\begin{proof} Indeed, assume that  $q\ge q_0^*$. It follows that $q>q_0(u)$ for each  $u\in H^1(\mathbb{R}^3)\setminus\{0\}$ and from item \eqref{i2} of Proposition \ref{fiberingvariation}  there holds $\inf_{t>0}\psi_{q,u}(t)=0$. Therefore $\mathcal{J}_q(u)\ge 0$.
	
	Now assume that $q<q_0^*$. From the definition of $q_0^*$, there exists $w\in H^1(\mathbb{R}^3)\setminus\{0\}$ such that $q<q_0(w)<q_0^*$. Threfore, from \eqref{i2} of Proposition \ref{fiberingvariation} we conclude that $\inf_{t>0}\psi_{q,w}(t)<0$ and hence there exists $t>0$ such that if $u:= tw$, it holds $\mathcal{J}_q(u)<0$.
\end{proof}

%Now we study the Nehari set $\mathcal{N}_q^0$.
%\begin{proposition}\label{N0ene} There exists a positive constant $m_q$ such that $\mathcal{J}_q(u)\ge m_q$ for each $u\in \mathcal{N}_q^0$.
%\end{proposition}
%\begin{proof} From the equations \eqref{extremal} with $t=1$ we have that 	
%	\begin{equation}\label{extremal1}
%	\left\{
%	\begin{aligned}
%	\|u\|^2+q^2\int \phi_u u^2-\|u\|_p^p &= 0, \\
%	\|u\|^2+3q^2\int \phi_u u^2-(p-1)\|u\|_p^p &= 0.
%	\end{aligned}
%	\right.
%	\end{equation}
%It follows from \eqref{extremal1} that $\mathcal{J}_q(u)=D_q \|u\|^2$ for each $u\in \mathcal{N}_q^0$ and some positive constant $D_q$. From the Proposition \ref{awayzero} the proof is completed.
%\end{proof}

Let us conclude this section with the following important result.
%Since $p\in(2,3]$ the functional $\mathcal J_{q}$ is not bounded below on $\mathcal N_{q}$.
%Nevertheless we have the following.

\begin{proposition}\label{N0ene} There exists a positive constant $m$ such that 
$$\forall q\in \mathbb R, u\in \mathcal{N}_q^0 :  \mathcal{J}_q(u)\ge m.$$
\end{proposition}
\begin{proof} From the equations \eqref{extremal} with $t=1$ we have that 	
	\begin{equation*}%\label{extremal1}
	\left\{
	\begin{aligned}
	\|u\|^2+q^2\int \phi_u u^2-\|u\|_p^p &= 0, \\
	\|u\|^2+3q^2\int \phi_u u^2-(p-1)\|u\|_p^p &= 0.
	\end{aligned}
	\right.
	\end{equation*}
It follows that $$q^{2}\int\phi_{u}u^{2}=\|u\|_{p}^{p} - \|u\|^{2} \quad \text{ and }\quad \|u\|_{p}^{p} = \frac{2}{4-p}\|u\|^{2}$$
so that $\mathcal{J}_q(u)=\frac{p-2}{4p}\|u\|^2$ for each $u\in \mathcal{N}_q^0$.
From \eqref{eq:bddaway}  the proof is completed.
\end{proof}

It is worth to point out that all that we have done in this section does not use the radial setting,
and clearly these results also holds in $H^{1}_{r}(\mathbb R^{3})$.

\section{Global Minima and (PS) Sequences for $\mathcal{J}_q$}\label{sec:PS}
In this section we prove the following result. Here and in the next section is fundamental 
to work with radial functions.
\begin{proposition}\label{MPS} There holds:
	\begin{enumerate}[label=(\roman*),ref=\roman*]
		\item\label{iMPS} for each $q\in (0,q_0^*)$ we have that $-\infty<\inf_{u\in H^1_{r}(\mathbb{R}^3)}\mathcal{J}_q(u)<0$;
		\item\label{iiMPS} for each $q> 0$, if $\{u_n\}\subset H_r^1(\mathbb{R}^3)$ is a sequence such that $\mathcal{J}'_q(u_n)\to 0$ as $n\to \infty$, then $\{u_n\}$ is convergent, up to subsequences.
	\end{enumerate}
\end{proposition}

%\begin{proposition}\label{min} For each $q\in (0,q_0^*)$ we have that $-\infty<\inf_{u\in H^1_r(\mathbb{R}^3)}\mathcal{J}_q(u)<0$.
%\end{proposition}
\begin{proof} 

Let us show \eqref{iMPS}.
 Indeed, from the Corollary \ref{zeroenergylevel} we know that $\inf_{u\in H^1_r(\mathbb{R}^3)}\mathcal{J}_q(u)<0$. We claim that $-\infty<\inf_{u\in H^1_{r}(\mathbb{R}^3)}\mathcal{J}_q(u)$. In fact, given $\varepsilon>0$ such that $D:=\frac{q^2}{16\pi}-\varepsilon^4>0$ , by  Proposition \ref{ine1}, we have that for each  $u\in H^1_r(\mathbb{R}^3)$ there holds
\begin{eqnarray}%\label{M1}
\mathcal{J}_q(u)&=&\frac{1}{4}\|\nabla u\|_2^2+\frac{1}{4}\|\nabla u\|_2^2+\frac{1}{2}\| u\|_2^2+\frac{q^2}{4}\int \phi_u u^2-\frac{1}{p}\|u\|_p^p \nonumber\\
                &\ge& \frac{1}{4}\|\nabla u\|_2^2+D\|\phi_u\|_\mathcal{D}^2+\frac{1}{2}\| u\|_2^2+\frac{\pi\varepsilon^2}{4}\|u\|_3^3-\frac{1}{p}\|u\|_p^p\nonumber\\
                &=&\frac{1}{4}\|u\|^{2}+D\|\phi_{u}\|_{\mathcal D}^{2}+\int f(u) \label{M2}
\end{eqnarray}
where
\begin{equation*}\label{ft}
f(t)=\frac{1}{4}t^2+\frac{\pi\varepsilon^2}{4}t^3-\frac{1}{p}t^p,\quad \forall\, t>0.
\end{equation*}
A simple analysis shows that $I:=\inf_{t>0}f(t)>-\infty$ and if $f(t)<0$ for some $t>0$, then $f^{-1}((-\infty,0))=(\alpha,\beta)$, where $0<\alpha<\beta<\infty$.
% Note from \eqref{M1} that 
%\begin{equation}\label{ft1}
%\mathcal{J}_q(u)\ge \frac{\|u\|^2}{4}+D_q\|\phi_u\|_\mathcal{D}^2+\int f(u),\quad \forall\, u\in H^1_r(\mathbb{R}^3). 
%\end{equation}

If $I\ge 0$, being $D_q>0$,  from  \eqref{M2} we conclude that $-\infty<\inf_{u\in H^1_r(\mathbb{R}^3)}\mathcal{J}_q(u)$. 

If $I<0$ then
\begin{equation}\label{eq:bo}
\mathcal J_{q}(u)\geq\frac{1}{4}\|u\|^{2}+D_{q}\|\phi_{u}\|_{\mathcal D}^{2}+I\,{\sl meas}(A)
\end{equation}
where $A=\{ x\in \mathbb R^{3}: u(x)\in (\alpha,\beta)\}$.
If there exists a sequence $\{u_n\}\subset H^1_r(\mathbb{R}^3)$ such that $\mathcal{J}_q(u_n)\to -\infty$ as $n\to \infty$,
then $\|u_{n}\|\to+\infty$. 
Moreover by \eqref{eq:bo}
%Define $A_n=\{x\in \mathbb{R}^3:\ u_{n}(x)\in (\alpha,\beta)\}$. Since $\mathcal{J}_q(u_n)\to -\infty$ as $n\to \infty$, 
we can assume without loss of generality that 
\begin{equation}\label{F1}
\frac{1}{4}\|u_n\|^2<|I| \,{\sl meas}(A_{n}),\quad \forall\, n\in \mathbb N.
\end{equation}
By the result of Strauss \cite{Strau} we know that there exist a positive constant $C$ such that
\begin{equation}\label{stra}
|u(x)|\le C|x|^{-1}\|u\|,\quad \forall\,u\in H^1_r(\mathbb{R}^3).
\end{equation}
Define $\rho_n=\sup\{|x|:\ x\in A_n\}$ and observe from the inequalities \eqref{F1} and \eqref{stra} that, for every $x\in \mathbb R^{3}$
with $|x|=\rho_{n}$ it holds,
\begin{equation*}
0< \alpha\le u_{n}(x)\le C\rho_n^{-1}\|u_n\|\le 2C\rho_n^{-1}(|I| {\sl meas}(A_n))^{1/2}, %,\ \forall\ x\in \mathbb{R}^3,\ |x|=\rho_n,
\end{equation*}
and hence, for some $C'>0$, we deduce
\begin{equation}\label{F2}
C'\rho_n\le {\sl meas}(A_n)^{1/2}.
\end{equation}
 Similar to the deduction of \eqref{F1}, we can assume without loss of generality that
\begin{equation*}
D\|\phi_{u_n}\|_\mathcal{D}^2<|I|{\sl meas}(A_n),\quad \forall\, n\in \mathbb N,
\end{equation*}
and hence, once that the function $(0,\infty)\ni t\mapsto (1-e^{-t/a})/t$ is decreasing we conclude that  
\begin{eqnarray*}
|I| {\sl meas}(A_n)&>&D\|\phi_{u_n}\|_\mathcal{D}^2 \\
          &=&\int\int\frac{1-e^{-\frac{|x-y|}{a}}}{|x-y|}u^2_n(x)u^2_n(y) \\
          &\ge& \int_{A_n}\int_{A_n}\frac{1-e^{-\frac{|x-y|}{a}}}{|x-y|}u^2_n(x)u^2_n(y) \\
          &\ge&\frac{1-e^{-\frac{2\rho_n}{a}}}{2\rho_n}\alpha^4 {\sl meas}(A_n)^2,
\end{eqnarray*}
which implies that 
\begin{equation}\label{F22}
\frac{|I|}{\alpha^{4}}\ge \frac{1-e^{-\frac{2\rho_n}{a}}}{2\rho_n}{\sl meas}(A_n),\quad \forall n\in \mathbb N,
\end{equation}
Observe from \eqref{F2} that $|A_n|\to \infty$ as $n\to \infty$ and from \eqref{F22} that $\rho_n\to \infty$ as $n\to \infty$. Combining \eqref{F2} with \eqref{F22} we obtain that
\begin{equation*}
C''\ge (1-e^{-\frac{2\rho_n}{a}})\rho_n,
\end{equation*}
for some $C''>0$,
which is clearly a contradiction and therefore  
\eqref{iMPS} is proved. %$-\infty<\inf_{u\in H^1_r(\mathbb{R}^3)}\mathcal{J}_q(u)$.

Let us show \eqref{iiMPS}.
%\begin{proposition}\label{PS} If $u_n\in H_r^1(\mathbb{R}^3)$ is a sequence such that $\mathcal{J}'_q(u_n)\to 0$ as $n\to \infty$, then $u_n$ is convergent.	
%\end{proposition}
From the convergence $\mathcal{J}'_q(u_n)\to 0$ as $n\to \infty$, we can assume without loss of generality that 
\begin{equation}\label{Q0}
\mathcal{J}'_q(u_n)[u_n]\le \|u_n\|,\ \forall n\in \mathbb N.
\end{equation}
On the other hand, from  Proposition \ref{ine1} we have that
\begin{equation*}
\|u_n\|_3^3\le \frac{1}{\pi}\left(\frac{1}{\varepsilon^2}\|\nabla u_n\|_2^2+\varepsilon^2\|\phi_{u_n}\|_\mathcal{D}^2\right),
\end{equation*}
where $\varepsilon>0$ is chosen now such that $\frac{q^2}{4\pi}-\frac{\varepsilon^4}{2}>0$.
It follows that, for all $n\in \mathbb N$: 
\begin{eqnarray}\label{Q1}
\mathcal{J}_q'(u_n)[u_n]&=&\|\nabla u_n\|^2+\|u_n\|_2^2+\frac{q^2}{4\pi}\|\phi_{u_n}\|_\mathcal{D}^2-\|u_n\|_p^p \nonumber\\
                      &\ge& \frac{1}{2}\|u_n\|^2+\frac{1}{2}\|u_n\|_2^2+\frac{\pi\varepsilon^2}{2}\|u_n\|_3^3-\frac{\varepsilon^4}{2}\|\phi_{u_n}\|_\mathcal{D}^2+\frac{q^2}{4\pi}\|\phi_{u_n}\|_\mathcal{D}^2-\|u_n\|_p^p \nonumber\\
                      &=& \frac{1}{2}\|u_n\|^2+\left(\frac{q^2}{4\pi}-\frac{\varepsilon^4}{2}\right)	\|\phi_{u_n}\|_\mathcal{D}^2+\frac{1}{2}\|u_n\|_2^2+\frac{\pi\varepsilon^2}{2}\|u_n\|_3^3-\|u_n\|_p^p \nonumber\\
                      &=& \frac{1}{2}\|u_n\|^2+\left(\frac{q^2}{4\pi}-\frac{\varepsilon^4}{2}\right)	\|\phi_{u_n}\|_\mathcal{D}^2+\int g(u_n),
\end{eqnarray}
where $g(t)=t^2/2+\frac{\pi\varepsilon^2}{2}t^3-t^p$ for $t>0$. We combine \eqref{Q0} with \eqref{Q1} to conclude that 
\begin{equation*}
\|u_n\|\ge  \frac{1}{2}\|u_n\|^2+\left(\frac{q^2}{4\pi}-\frac{\varepsilon^4}{2}\right)	\|\phi_{u_n}\|_\mathcal{D}^2+\int g(u_n).
\end{equation*}
We conclude as in the proof of \eqref{iMPS} that $\{u_n\}$ is bounded. Once we know that $\{u_n\}$ is bounded, standard arguments (observe that the analogous of \cite[Lemma 2.1]{Ruiz} is valid) produce a convergent subsequence.
% and without loss of generality we assume that the whole sequence does converge.
\end{proof}

\begin{remark}\label{rem:qn}
Note that \eqref{iiMPS} in the Proposition \ref{MPS} can be extended in the following way:
if $q_{n} \to q$ and  if $\{u_n\}\subset H_r^1(\mathbb{R}^3)$ is a sequence such that $\mathcal{J}'_{q_{n}}(u_n)\to 0$ as $n\to \infty$, then $\{u_n\}$ is convergent, up to subsequences.
This follows due to the smooth dependence of $\mathcal J_{q}'$ on $q$.
\end{remark}

%Now we can prove the Lemma \ref{MPS}:
%\begin{proof}[Proof of the Lemma \ref{MPS}] It follows from the Propositions \ref{min} and \ref{PS}.	
%\end{proof}
\section{Existence of Two radial Solutions} \label{sec:final}
In this section we prove item 2. of Theorem \ref{T1}.
%\begin{proposition}\label{twoso} There exists $\varepsilon>0$ such that for each $q\in (0,q_0^*+\varepsilon)$ the functional $\mathcal{J}_q$ has two critical points $u_q\neq 0$ and $w_q\neq 0$. More specifically,
%$w_q$ is a Mountain Pass critical point and
%	\begin{enumerate}
%		\item[i)] if $q\in (0,q_0^*]$, then $u_q$ is a global minimum, 
%		\item[ii)] if $q\in (q_0^*,q_0^*+\varepsilon)$, then $u_q$ is a local minimum.
%			\end{enumerate}
%		Furthermore
%		\begin{enumerate}
%		\item[iii)] $\mathcal{J}_q(u_q)<0$ if $q\in (0,q_0^*)$, $\mathcal{J}_q(u_{q_0^*})=0$ and $\mathcal{J}_q(u_q)>0$ if $q\in (q_0^*,q_0^*+\varepsilon)$, while $\mathcal{J}_q(w_q)>0$ and $\mathcal{J}_q(w_q)>\mathcal{J}_q(u_q)$ for each $q\in (0,q_0^*+\varepsilon)$.
%	\end{enumerate}
%\end{proposition}

\begin{proposition}\label{beforeq_0} For each $q\in (0,q_0^*)$ there exists a global minimum $u_q$ such that $\mathcal{J}_q(u_q)<0$.
\end{proposition}
\begin{proof} It follows from the Proposition \ref{MPS} and Ekeland's Variational Principle.
\end{proof}
Now we prove the existence of a local minimizer for $\mathcal{J}_q$ when $q$ is near $q_0^*$. To do so, we first prove the existence of a global minimizer for the functional $\mathcal{J}_{q_0^*}$.
\begin{corollary}\label{Mq_0} 
There exists a global minimizer $u_{q_0^*}\neq 0$ of $\mathcal J_{q_{0}^{*}}$ 
such that %$\mathcal J_{q_{0}^{*}}'(u_{q_0^*})=0$ and 
$ \mathcal{J}_{q_0^*}(u_{q_0^*})=0$.
\end{corollary}
\begin{proof} Indeed, suppose that $q_n\uparrow q_0^*$ as $n\to \infty$. From the Proposition \ref{beforeq_0}, for each $n$, there exists $u_n:= u_{q_n}$ such that $u_n$ is a global minimum for $\mathcal{J}_{q_n}$ and $\mathcal{J}_{q_n}(u_n)<0$. It follows that $\mathcal{J}'_{q_n}(u_n)=0$ for each $n$ and, 
being all the Nehari manifolds bounded away from zero uniformly in $q$, by \eqref{eq:bddaway} it is
$\|u_n\|\ge \widetilde C$ for each $n$. From \eqref{iiMPS} in Proposition \ref{MPS}
(see Remark \ref{rem:qn}) we conclude that $u_n\to u\neq 0$. By $\mathcal{J}_{q_n}(u_{n})<0$ for each $n$, we conclude that $\mathcal{J}_{q_0^*}(u)\le 0$ and from  Corollary \ref{zeroenergylevel} it follows that $\mathcal{J}_{q_0^*}(u)=0$. Then it is sufficient to set $u_{q_0^*}:= u$ and  the proof is completed.
\end{proof}
\begin{remark}\label{RMK} From the definition of $q_0^*$ and the Corollary \ref{Mq_0} it follows that $q_0(u_{q_0^*})=q_0^*$. Moreover $q_0^*<q(u_{q_0^*})$.
\end{remark}
For $q>0$, define
\begin{equation*}%\label{MINN}
\widehat{\mathcal{J}}_q:=\inf \left\{\mathcal{J}_q(u):\ u\in \mathcal{N}_q^+\cup \mathcal{N}_q^0 \right\}.
\end{equation*}
Observe that 
\begin{equation}\label{eq:infimo}
\widehat{\mathcal{J}}_q=\inf_{u\in H^1_r(\mathbb{R}^3)}\mathcal{J}_q(u)\quad \forall q\in (0,q_0^*]
\end{equation}
 and from  Corollary \ref{zeroenergylevel} there holds  $\widehat{\mathcal{J}}_q\ge 0$ for $q \geq q_0^*$.
\begin{proposition}\label{enerynear} Given $\delta>0$, there exists $\varepsilon>0$ such that for each $q\in (q_0^*,q_0^*+\varepsilon)$ there holds $\widehat{\mathcal{J}}_q<\delta$.
\end{proposition}
\begin{proof} Indeed, let $u_{q_0^*}\in \mathcal N_{q}^{+}$ be given as in  Corollary \ref{Mq_0}. Observe that if $q\downarrow q_0^*$, then $\mathcal{J}_q(u_{q_0^*})\to \mathcal{J}_{q_0^*}(u_{q_0^*})=0$. Moreover, once $q_0^*<q(u_{q_0^*})$, it follows that there exists $\varepsilon_1>0$ such that $q_0^*+\varepsilon_1<q(u_{q_0^*})$. From  Proposition \ref{fiberingmaps} and \eqref{i2} in Proposition \ref{fiberingvariation}, for each $q\in (q_0^*,q_0^*+\varepsilon_1)$, there exists $t_q^+(u_{q_0^*})$  such that $t_q^+(u_{q_0^*})u_{q_0^*}\in \mathcal{N}_q^+$. Note that $t_q^+(u_{q_0^*})\to 1$ as $q\downarrow q_0^*$ and therefore
	\begin{equation*}
	\mathcal{J}_q(u_{q_0^*})\le \mathcal{J}_q(t_q^+(u_{q_0^*})u_{q_0^*})\to \mathcal{J}_{q_0^*}(u_{q_0^*})=0,\ q\downarrow q_0^*.
	\end{equation*}
If $\varepsilon_2>0$ is choosen in such a way that $\mathcal{J}_q(t_q^+(u_{q_0^*})u_{q_0^*})<\delta$ for each $q\in (q_0^*,q_0^*+\varepsilon_2)$, then we set $\varepsilon=\min\{\varepsilon_1,\varepsilon_2\}$ and the proof is completed.
\end{proof}
Let us recall that by
Proposition \ref{MPGAN}, there exist positive constants $\rho,M$ such that $\mathcal{J}_q(u)\ge M$ for each $\|u\|=\rho$. We can assume without loss of generality that $\rho<\widetilde C$, where $\widetilde C$ is 
such that
$$\text{ for all }q\in \mathbb R, \ u\in \mathcal N_{q} : \|u\|\geq \widetilde C,  $$
(see \eqref{eq:bddaway}).
%given by the Proposition \ref{awayzero}.	

We choose $\delta>0$  in the Proposition \ref{enerynear} in such a way that 
\begin{equation}\label{eq:delta}
\delta<\min\{M,m\},
\end{equation}
where $m$ is the 
positive constant such that, by  Proposition \ref{N0ene},
 $$\forall q\in \mathbb R, u\in \mathcal{N}_q^0 :  \mathcal{J}_q(u)\ge m.$$
Let $\varepsilon>0$ be as in Proposition \ref{enerynear} in  correspondence of the above fixed
$\delta>0$.

\begin{proposition}\label{EKE1} There holds
	\begin{equation*}
	\inf\left\{\mathcal{J}_q(u):\ \|u\|\ge \rho\right\}=\widehat{\mathcal{J}}_q,\quad \forall\, q\in (q_0^*,q_0^*+\varepsilon).
	\end{equation*}
\end{proposition}
\begin{proof} First observe from the inequality $\rho<\widetilde C$ that $\inf\{\mathcal{J}_q(u):\|u\|\ge \rho\}\le \widehat{\mathcal{J}}_q$. We claim that the equality holds. Indeed, by one hand, if the fiber map $\psi_{q,u}$ satisfies \eqref{i2fibering} or \eqref{i3fibering} of the Proposition \ref{fiberingmaps}, then $\inf_{t>\rho} \psi_{q,u}(t)=M$. On the other hand, if the fiber map $\psi_{q,u}$ satisfies \eqref{i1fibering} of the Proposition \ref{fiberingmaps}, then $\inf_{t>\rho} \psi_{q,u}(t)\ge\widehat{\mathcal{J}}_q $. Once 
	$M>\delta >\widehat{\mathcal{J}}_q$ the proof is completed.
\end{proof}
\begin{corollary}\label{LocM} %There exists $\varepsilon>0$ such that f
For each $q\in (q_0^*,q_0^*+\varepsilon)$ there holds: there exists $u_q\in \mathcal{N}_q^+$ such that $\mathcal{J}_q(u_q)=\widehat{\mathcal{J}}_q$. In particular $\mathcal J_{q}(u_{q})>0$ and  $\|u_q\|\ge \widetilde C>\rho$.
\end{corollary}
\begin{proof}
%Choose $\varepsilon_1>0$ in such a way that $m_q>\delta_q$ for each $q\in (q_0^*,q_0^*+\varepsilon_1)$, where %$m_q$ is given by the Proposition \ref{N0ene}. From the Proposition \ref{EKE1} we know that 
%	\begin{equation}\label{TT}
%	\inf\{\mathcal{J}_q(u):\ \|u\|\ge \rho\}=\widehat{\mathcal{J}}_q,\quad \forall\, q\in (q_0^*,q_0^*+\varepsilon_q).
%	\end{equation}
%	Define $\varepsilon=\min\{\varepsilon_1,\varepsilon_q\}$ and observe
%	 From the inequality $M>\delta>\widehat{\mathcal{J}}_q$
	Fix $q\in (q_0^*,q_0^*+\varepsilon)$ and let $\{u_n\}\subset \mathcal N_{q}^{+}\cup \mathcal N_{q}^{0}$ be a minimising sequence for $\widehat{\mathcal J}_{q}<\delta$ by Proposition \ref{enerynear}.  Since $m>\delta$  and, by Proposition \ref{N0ene}, $\mathcal J_{q} (u)\geq m$ on $\mathcal N_{q}^{0}$, we can assume that $\{u_{n}\}\subset \mathcal N_{q}^{+}$ and hence, by the Ekeland's Variational Principle, 
also that $\mathcal{J}'_q(u_n)\to 0$.
	  %Therefore, once \eqref{TT} holds for each $q\in (q_0^*,q_0^*+\varepsilon)$, we can apply the Ekeland's Variational Principle to produce a sequence $\{u_n\}$ such that $\mathcal{J}_q(u_n)\to \widehat{\mathcal{J}}_q$ and $\mathcal{J}'_q(u_n)\to 0$. 
	  We conclude from the Proposition \ref{MPS}  that $u_n\to u$ in $H_r^1(\mathbb{R}^3)$ with $\|u\|\ge \widetilde C> \rho$. 
	Setting $u_q:= u$ clearly we obtain that  $u_q\in \mathcal{N}_q^+$ and $\mathcal{J}_q(u_q)=\widehat{\mathcal{J}}_q$. Due to the definition of $q_0^*$ and the fact that $q>q_0^*$, we conclude that $\mathcal J_{q}(u_{q})>0$.
\end{proof}
We observe two properties of  the function $(0,q_0^*+\varepsilon)\ni q\mapsto\widehat{\mathcal{J}}_q$.
\begin{lemma}\label{contidecre} The function $(0,q_0^*+\varepsilon)\ni q\mapsto\widehat{\mathcal{J}}_q$ is increasing and continuous.
\end{lemma}
\begin{proof} %We first prove that $(0,q_0^*+\varepsilon)\ni q\mapsto\widehat{\mathcal{J}}_q$ is decreasing. 
Indeed, suppose that $q<q'$. From  Proposition \ref{beforeq_0}, Corollary \ref{Mq_0} and the Corollary \ref{LocM}
 there exists $u_{q'}$ such that $\widehat{\mathcal{J}}_{q'}=\mathcal{J}_{q'}(u_{q'})$. 
 
 If $q'\in (q_0^*,q_0^*+\varepsilon)$, from the Corollary \ref{LocM} it is $\|u_{q'}\|\ge \widetilde C>\rho$ and hence from the Proposition \ref{EKE1} we obtain 
	\begin{equation*}
	\widehat{\mathcal{J}}_{q}\le \mathcal{J}_{q}(u_{q'})<\mathcal{J}_{q'}(u_{q'})=\widehat{\mathcal{J}}_{q'}.
	\end{equation*}
	If $q'\in (0,q_{0}^{*}]$ the lemma follows by \eqref{eq:infimo}.

\medskip

Now we prove that $(0,q_0^*+\varepsilon)\ni q\mapsto\widehat{\mathcal{J}}_q$ is continuous. In fact, suppose that $q_n\uparrow q\in (0,q_0^*+\varepsilon)$. From  Proposition \ref{beforeq_0}, Corollary \ref{Mq_0} and  Corollary \ref{LocM}, for each $n$, there exists $u_n:= u_{q_n}$ such that 	$\widehat{\mathcal{J}}_{q_n}= \mathcal{J}_{q_n}(u_n)$. Similar to the proof of  Corollary \ref{Mq_0} we may assume that $u_n\to u\neq 0$. 

As before, if $q\in (0,q_{0}^{*}]$ the lemma holds due to   \eqref{eq:infimo}. 

If $q\in (q_0^*,q_0^*+\varepsilon)$ observe from  Corollary \ref{LocM} that $\|u\|>\rho$. We claim that $\widehat{\mathcal{J}}_{q_n}\to \widehat{\mathcal{J}}_q$ as $n\to \infty$. Indeed, once $(0,q_0^*+\varepsilon)\ni q\mapsto\widehat{\mathcal{J}}_q$ is increasing, we can assume that $\widehat{\mathcal{J}}_{q_n}< \widehat{\mathcal{J}}_q$ for each $n$ and $\widehat{\mathcal{J}}_{q_n}\to \mathcal{J}_q(u) \le\widehat{\mathcal{J}}_q$ as $n\to \infty$, which implies that $\mathcal{J}_q(u) =\widehat{\mathcal{J}}_q$.

Now suppose that $q_n\downarrow q\in (0,q_0^*+\varepsilon)$. Once $(0,q_0^*+\varepsilon)\ni q\mapsto\widehat{\mathcal{J}}_q$ is increasing, we can assume that $\widehat{\mathcal{J}}_{q}<\widehat{\mathcal{J}}_{q_{n}}$ and $\widehat{\mathcal{J}}_q\leq \lim_{n\to \infty}\widehat{\mathcal{J}}_{q_n}$. Choose $u_q$ such that $\widehat{\mathcal{J}}_q=\mathcal{J}_q(u_q)$ ($\|u\|>\rho$ in case $q\in (q_0^*,q_0^*+\varepsilon)$) and observe that $\widehat{\mathcal{J}}_q\le \lim_{n\to \infty}\widehat{\mathcal{J}}_{q_n}\le\lim_{n\to \infty}\mathcal{J}_{q_n}(u_q)=\widehat{\mathcal{J}}_q$.
\end{proof}

Now we turn our attention to the second solution.

Let $q\in (0,q_{0}^{*})$.
 As a consequence of the Corollary \ref{zeroenergylevel} we have that $$\Gamma_{q}=\{\gamma\in C([0,1],H_r^1(\mathbb{R}^3)):\ \gamma(0)=0,\ \mathcal{J}_{q}(\gamma(1))<0 \}$$ is non-empty. Define the Mountain Pass level
\begin{equation*}%\label{MPE}
c_{q}:= \inf_{\gamma\in \Gamma_{q}}\max_{t\in [0,1]}\mathcal{J}_{q}(\gamma(t))> 0.
\end{equation*}

By Proposition \ref{MPGAN} and Proposition \ref{MPS} we deduce the following
\begin{proposition}\label{MPG} For each $q\in (0,q_0^*)$ there exists $w_q\in H^1_r(\mathbb{R}^3)\setminus\{0\}$ such that $\mathcal{J}_q(w_q)=c_q$ and $\mathcal{J}'_q(w_q)=0$. In particular $\mathcal J_{q}(w_{q})>\mathcal J_{q}(u_{q})\in(-\infty,0).$
\end{proposition}
%\begin{proof} Indeed, we combine the Proposition \ref{MPGAN} with the inequality $M_q>0$ to obtain a mountain %pass geometry for the functional $\mathcal{J}_q$. The proof follows from the Proposition \ref{MPS}.
%\end{proof}

Let now $q\in (q_0^*,q_0^*+\varepsilon)$,
where $\varepsilon>0$ is the one fixed in correspondence of $\delta$ in \eqref{eq:delta}. 
%given as in the Corollary \ref{LocM} (\textcolor{red}{ com $q>q_{0}^{*}$}). 
Let  $u_q\in \mathcal{N}_q^+$ (by Proposition \ref{LocM}) such that $\mathcal{J}_q(u_q)=\widehat{\mathcal{J}}_q$. Define
\begin{equation*}%\label{MPA}
d_q=\inf_{\gamma\in \Gamma_q}\max_{t\in [0,1]}\mathcal{J}(\gamma(t)),
\end{equation*}
	where $\Gamma_{q}=\{\gamma\in C([0,1],H_r^1(\mathbb{R}^3)):\ \gamma(0)=0,\ \gamma(1)=u_q \}$.
\begin{proposition}\label{MPCP} For each $q\in (q_0^*,q_0^*+\varepsilon)$ there exists $w_q\in H^1_r(\mathbb{R}^3)\setminus\{0\}$ such that $\mathcal{J}_q(w_q)=d_q$ and $\mathcal{J}'_q(w_q)=0$. In particular $\mathcal J_{q}(w_{q})>\mathcal J_{q}(u_{q})$.
\end{proposition}
\begin{proof} Indeed, we combine  Proposition \ref{MPGAN} with the inequality $M>\delta\ge \widehat{\mathcal{J}}_q =\mathcal{J}_q(u_q)$ to obtain a Mountain Pass Geometry for the functional $\mathcal{J}_q$. The proof follows from \eqref{iiMPS} in Proposition \ref{MPS}.
\end{proof}

\medskip

Now we conclude the proof of item 2. of Theorem  \ref{T1}.

Let $\varepsilon$ be given as in the Proposition \ref{EKE1}. The existence of  the minimum
$u_q$ follows from Proposition \ref{beforeq_0}, Corollary \ref{Mq_0}, Corollary \ref{LocM}.
The existence of a Mountain Pass critical point $w_{q}$ satisfying $\mathcal J_{q}(w_{q})>\max\{ 0, \mathcal J_{q}(u_{q})\}$ follows by Proposition \ref{MPG} and Proposition \ref{MPCP}.
That $u_{q}$ and $w_{q}$ are actually critical points of $\mathcal J_{q}$ follows by Proposition \ref{minnehari}.

%\bigskip
%
%$,w_q$ for each $q\in (0,q_0^*+\varepsilon)$ and the proof of iii) follows from the  (all fine), , Proposition \ref{contidecre}, Proposition \ref{MPG} and the Proposition \ref{MPCP}. 
%
%The proof of i) follows from the Proposition \ref{beforeq_0}, Corollary \ref{Mq_0}, Proposition \ref{MPG} and the Proposition \ref{MPCP}.
%
%The proof of ii) follows from the Corollary \ref{LocM} and the Proposition \ref{MPCP}. 

%------------------------------------------------------
%------------------------------------------------------

\end{document}